\documentclass[12pt]{amsart}
\usepackage{amssymb, latexsym}
\usepackage{amscd}
\usepackage{eucal}
\usepackage{amsfonts}
\usepackage{latexsym}
\usepackage{amsthm}
\usepackage[english]{babel}
\newtheorem{tw}{Theorem}[section]
\newtheorem{lem}[tw]{Lemma}
\newtheorem{wn}[tw]{Corollary}
\newtheorem{fakt}[tw]{Fact}
\newtheorem{df}[tw]{Definition}

\newcommand{\dx}{\mathrm{d}x}

\newcommand{\e}{\varepsilon}
\newcommand{\f}{\mu}
\newcommand{\lm}{\lambda}
\newcommand{\D}{\Delta}

\newcommand{\Rz}{\mathbb{R}}

\numberwithin{equation}{section}

\begin{document}
\title{Perturbations of isometries between Banach spaces}
\author{Rafa\l ~G\'orak}
\address{Rafa\l ~G\'orak\\ Technical University of Warsaw \\
Pl. Politechniki 1\\  00-661 Warszawa \\ Poland}
\date{\today}

\maketitle

\begin{abstract}
We prove a very general theorem concerning the estimation of the
expression \mbox{$\|T(\frac{a+b}{2}) - \frac{Ta+Tb}{2}\|$} for
different kinds of maps $T$ satisfying some general perurbated
isometry condition. It can be seen as a quantitative
generalization of the classical Mazur-Ulam theorem. The estimates
improve the existing ones for bi-Lipschitz maps. As a consequence
we also obtain a very simple proof of the result of Gevirtz which
answers the Hyers-Ulam problem and we prove a non-linear
generalization of the Banach-Stone theorem which improves the
results of Jarosz and more recent results of Dutrieux and Kalton.
\end{abstract}
\section{Introduction}

The aim of this paper is to prove a very general theorem (Theorem
\ref{tw coarse mazur ulam}) that will allow us to obtain several
facts concerning approximate preservation of midpoints by
different kinds of maps with perturbated isometry condition. Let
us define the main notion of this paper:
\begin{df}
Let $T :E \mapsto F$ be a function between two metric spaces
$(E,d_E)$ and $(F,d_F)$. Assume that there is a function
\mbox{$\f: \Rz_+ \mapsto \Rz_+$} (where $\Rz_+ = \{x \in \Rz; \; x
\geq 0\}$) which is non-decreasing and such that the following
conditions hold:
\begin{itemize}
\item[(i)] $T$ is a bijection. \item[(ii)] $d_F(Tx,Ty) \leq
\f(d_E(x,y))$ and $d_E(T^{-1}f,T^{-1}g) \leq \f(d_F(f,g))$ for all
$x,y \in E$ and $f,g \in F$.
\end{itemize}
Then $T$ is called a $\f$-isometry.
\end{df}
\noindent In our article we consider (except Corollary \ref{wn o
sieciach}) \mbox{$\f$-isometries} between Banach spaces only. It
should be noticed that following \cite{LinSzan} for a given map
\mbox{$T:E \mapsto F$} we can easily find the optimal $\f$ which
is $\f(t)=t+\e_T(t)$ where

$$ \e_T(t)=\sup \{\big{|}\|Tx-Ty\|-\|x-y\|\big{|}\;: \|x-y\| \leq t \textrm{ or
} \|Tx-Ty\| \leq t \}.$$ Lindenstrauss and Szankowski consider
maps $T$ that are surjective but not necessarily injective as
$\f$-isometries. However they observed that one can easily reduce
the considerations to the bijective case when $t \rightarrow
\infty$:

\begin{fakt}\label{surjekcja a bijekcja}
Let $T:E \mapsto F$ be a surjective map between Banach spaces $E$
and $F$, respectively. If $\e_T :\Rz_+ \mapsto \Rz_+$ is well
defined ($\forall t \in \Rz_+ \; \e_T(t)< \infty$) and $\exists
\delta_0>0$ $\frac{\e_T(\delta_0)}{\delta_0}<1$ then there exists
a bijection $\widetilde{T}:E \mapsto F$ such that:

$$\forall x \in E \; \|Tx-\widetilde{T}x\| \leq
2\delta_0+2\e_T(\delta_0).$$ Hence  $\widetilde{T}$ is a
$\f$-isometry for $\f(t)=t+ \e_T(t)+4\delta_0+4\e_T(\delta_0)$. In
particular $\e_{\widetilde{T}}(t)\sim \e_T(t)$ as $t \rightarrow
\infty$ (if only $\e_T(t) \rightarrow \infty$).
\end{fakt}
\begin{proof}
For the sake of completeness we sketch the proof. Let us consider
the maximal set $A \subset E$ such that all the points are in the
distance at least $\delta_0$ from each other. Then for every $a
\neq b$, $a,b \in A$ we have $\delta_0-\e_T(\delta_0) \leq
\|Ta-Tb\|$, hence $T|A$ is injective. Moreover $T(A)$ is a
$\delta_0+\e_T(\delta_0)$ dense in $F$ (that is the distance of
every element of $F$ from $T(A)$ is not greater than
$\delta_0+\e_T(\delta_0)$). This shows that the density character
of $E$ and $F$ are equal. Now it is easy to construct a
decomposition of $E=\dot{\bigcup}_{a \in A} E_a$ and
$F=\dot{\bigcup}_{a \in A} F_a$ such that for all $a \in A$:
\begin{itemize}
\item[(1)] $a \in E_a$, $Ta \in F_a$;

\item[(2)] $|E_a|=|F_a|$;

\item[(3)] $\textrm{diam} E_a \leq \delta_0$ and $\textrm{diam}
F_a \leq \delta_0 + \e_T(\delta_0)$.
\end{itemize}
By the standard set theoretical reasoning we can extend $T|A$ to
the required $\f$-isometry $\widetilde{T}:E \mapsto F$.
\end{proof}
\noindent Hence in further considerations we stick to the notion
of $\f$-isometry as it provides sufficient generality and by
considering bijective maps we avoid some easy but rather technical
problems.

 When considering the $\f$-isometry $T$ one should rather
think that $T$ is not necessarily the perturbated isometry (since
it may easily happen that there is no isometry to be perturbated)
but $T$ satisfies the perturbated isometry condition. Hence the
following natural question arises: "How can you perturbate the
definition of an isometry between Banach spaces so that its
existence implies the existence of an isometry?". If the answer to
the above question is positive then another one can be asked: "How
far is the perturbated isometry from an isometry?" It appears that
Lindestrauss and Szankowski in \cite{LinSzan} answered these
questions for the class of all Banach spaces and for all
$\f$-isometries. However one can investigate the above problems
for some subclasses of Banach spaces (such as function spaces
which leads to generalizations of the Banach-Stone theorem).

Let us discuss now, in more details, some examples of
$\f$-isometries for different functions $\f$ and the results
related to both questions asked above. Let $T$ be a $\f$-isometry
between Banach spaces $E$ and $F$. If $\f(t)=t$ then $T$ is just
an isometry. Let us consider now $\f(t)=t+L$ for some constant $L
\geq 0$. Such maps are called $L$-isometries. More generally
$L$-isometry $T$ is a surjective map between Banach spaces for
which $\e_T(t) \leq L$. But as we have already noticed, Fact
\ref{surjekcja a bijekcja} allows us to reduce considerations to
the bijective case (see
 Corollary \ref{Hyers-Ulam} where we show how it is done). Hyers and Ulam asked whether
$L$-isometries are close to isometries. The question was answered
positively for all pairs of Banach spaces $E$ and $F$ by Gevirtz
in \cite{Gev1} (let us say that $L$ can be as large as we please).

Szankowski and Lindenstrauss gave a complete characterization of
such $\f$-isometries whose existence implies the existence of an
isometry. More precisely:
\begin{tw}\label{LinSzan}
Let $T :E \mapsto F$ be a $\f$-isometry between Banach spaces $E$
and $F$ where $\f(t)=t+\e_T(t)$, $T(0)=0$ such that the condition
\mbox{$\int^{\infty}_1 \frac{\e_T(t)}{t^2} dt < \infty$} is
satisfied. Then there exists an isometry $I:E \mapsto F$ such that

$$\|Tx-Ix\| = o(\|x\|) \textrm{ as } \|x\| \rightarrow \infty.$$
\noindent Moreover the result is sharp (see \cite{LinSzan} for
more details) in the case when $E$ and $F$ are general Banach
spaces.
\end{tw}

Let us consider now $\f(t)=Mt$. In this case $T$ is a bi-Lipschitz
map (or Lipschitz equivalence). It means that distances between
points are perturbated  according to the inequalities
$$\frac{1}{M}{\|x-y\|} \leq \|Tx-Ty\| \leq M{\|x-y\|}
\textrm{ for all } x,y \in E.$$ Obviously if $M = 1$ then $T$ is
just an isometry. Let us look at
 the case when $M \searrow 1$. Unfortunately, no matter how close to one $M$
is, we cannot guarantee the existence of an isometry between
general Banach spaces $E$ and $F$. Clearly $\int^{\infty}_1
\frac{(M-1)t}{t^2} dt = \infty$ ($M >1$) hence you can find in
\cite{LinSzan} a construction of Banach spaces $E$ and $F$ that
are $\f$-isometric for $\f(t)=Mt$ but they are not isometric.
However, for some particular class of Banach spaces $E$ and $F$
one can obtain some interesting positive results even for more
general case that is when $\f(t)=Mt+L$ (maps that are bi-Lipschitz
for large distances). Indeed let us consider $E=C_0(X)$ and
$F=C_0(Y)$, the spaces of continuous real valued functions
vanishing at $\infty$ on locally compact spaces $X$ and $Y$,
respectively. Spaces $C_0(X)$ and $C_0(Y)$ are endowed with the
$\sup$ norms. It appears that in this case one can obtain more
than Theorem \ref{LinSzan}:

\begin{tw}\label{Banach-Stone ogolnie}
Let $T :C_0(X) \mapsto C_0(Y)$ be a $\f$-isometry, where $X$ and
$Y$ are locally compact spaces, $\f(t)=Mt+L$ ($M \geq 1$, $L \geq
0$) and $T(0)=0$. Then there exists an absolute constant $M_0>1$
and functions $\delta : [1,\infty) \mapsto \Rz_+$, $\D : \Rz^2_+
\mapsto \Rz_+$ such that whenever $M<M_0$ then there exists an
isometry $I :C_0(X) \mapsto C_0(Y)$ such that
\begin{equation}\label{odleglosc od izometrii dla przestrzeni funkcyjnych}
 \|Tf-If\| \leq \delta(M) \|f\| + \D(M,L) \textrm{ for all } f \in
C_0(X).
\end{equation}
Moreover, $\D(M,0)=0$ and $\lim_{M \rightarrow 1^+}\delta(M)=0$.
In particular, from the Banach-Stone theorem, the spaces $X$ and
$Y$ are homeomorphic. It is known that $M_0 \leq \sqrt{2}$ and the
equality holds if we assume additionally that $T$ is linear (see
\cite{Cam} and \cite{Gor} for the discussion).
\end{tw}

The first of such results was obtained by Jarosz in \cite{Jarosz}
but for $L=0$ only. However the value of $M_0$ which he obtains is
very close to 1 as well as the function $\delta$ is far from being
optimal ($\delta(M)=O((M-1)^{0.1})$ as $M \searrow 1$ and $\D(M,0)
= 0$ in his result). Later Dutrieux and Kalton in \cite{DutKal}
obtained the value of $M_0=\sqrt{\frac{17}{16}}$ (in their
language the condition $M < M_0$ can be seen as the inequality
$d_N(C_0(X),C_0(Y))< M^2_0$) but they do not provide any
estimation like (\ref{odleglosc od izometrii dla przestrzeni
funkcyjnych}) (this time $L$ can be positive). Finally the author
in \cite{Gor} improved the constant to $M_0=\sqrt{\frac{6}{5}}$
and showed that $\delta(M)=26(M-1)$. Moreover $\D(M,0)=0$ hence
the result improved both, the constant $M_0$ obtained in
\cite{DutKal} and the function $\delta$ obtained in \cite{Jarosz}
as well as showed the existence of $\delta$ and $\Delta$ if $L>0$.
However, the proof works only for $X$ and $Y$ {\bf compact} and it
is not that easy to extend it to the locally compact case. We will
do this in the last section of this paper by applying the main
result of Section 2.

It appears that in the proofs of most of the above results the
estimation of \mbox{$\|T(\frac{a+b}{2}) - \frac{Ta+Tb}{2}\|$} is
crucial and far from being obvious. Moreover the results
estimating this expression can be regarded as generalizations of
the Banach-Mazur theorem so in some sense they are of independent
interest. We deal with this problem in the next section.

\section{Approximate Preservation of midpoints by $\f$-isometries}
We present here a very general method of estimating
\mbox{$\|T(\frac{a+b}{2}) - \frac{Ta+Tb}{2}\|$} for
$\f$-isometries $T$. It should be mentioned that some results of
this kind are already obtained in \cite{LinSzan} (in fact this is
the most demanding part of the article) . However the method
presented here has several important advantages. First of all it
has astonishingly simple proof and it covers the result of Gevirtz
(Corollary \ref{Hyers-Ulam}) which answers the famous Hyers-Ulam
problem (the proofs in the original paper \cite{Gev1} or in the
survey paper of Rassias \cite{Rassias} are clearly more
complicated). Secondly, applying our result for $\f$-isometries
where $\f(t)=Mt+L$, we obtain new and elegant estimates (they are
interesting even in the Lipschitz case that is when $L=0$). This
will allow us to prove new results concerning the nonlinear
version of the Banach-Stone theorem. Finally, although our theorem
does not cover the result of Lindenstrauss and Szankowski in full
generality, it gives their result for particular functions
$\f(t)=t+\e(t)$ such as $\f(t)=t+t^{\alpha}$ where $\alpha \in
[0,1)$ (see Section 4). It is very tempting (due to the simplicity
of the prove below) to investigate whether Theorem \ref{tw coarse
mazur ulam} gives us the result from \cite{LinSzan} in full
generality.

Before we formulate and prove the main result let us say that the
idea of it comes from a very beautiful proof of the classical
Mazur-Ulam theorem due to V\"{a}is\"{a}l\"{a} (see
\cite{Vaisala}).
\begin{tw}\label{tw coarse mazur ulam}
Let $T :E \mapsto F$ be a $\f$-isometry between two normed spaces
$(E,\|.\|_E)$ and $(F,\|.\|_F)$. Assume that $\f: \Rz_+ \mapsto
\Rz_+$ is such that $\f(t)/2 \leq \f(t/2)$. Then for all $a,b \in
E$ and $n \in \mathbb{Z}_+$:
$$\|T(\frac{a+b}{2})-\frac{Ta+Tb}{2}\|_F \leq \f^{\circ (2^{n+1}-1)}(\frac{\|a-b\|_E}{2^{n+1}})$$
where $\f^{ \circ n} = \f \circ \f \circ \ldots \f$ ($\f$ composed
$n$ times).
\end{tw}
\begin{proof}
Let us consider the set $W_E(\f)$ consisting of all maps $T$ that
are $\f$-isometries on $E$ and moreover, let $\textrm{Im}T$ be a
normed space. Fix $a$,$b$ in the space $E$ and set
$z=\frac{a+b}{2}$. Denote:
$$\lm(\f) = \sup \{\|Tz-\frac{Ta+Tb}{2}\|_F \textrm{ \;}|  \; T \in W_E(\f) \; , F=\textrm{Im}T\}.$$
Let us observe that for $T \in W_E(\f)$ we have:
\begin{eqnarray*}
\|Tz-\frac{Ta+Tb}{2}\|_F & \leq &
\frac{1}{2}(\|Tz-Ta\|_F+\|Tz-Tb\|_F) \leq \\ & \leq & \frac{1}{2}
(2\f(\frac{\|a-b\|_E}{2})) = \f(\frac{\|a-b\|_E}{2}).
\end{eqnarray*}
Hence
\begin{equation}\label{oszacowanie lambda}
\lm(\f) \leq \f(\frac{\|a-b\|_E}{2})
\end{equation}
and one can see that $\lm(\f)$ is finite. For some $T \in W_E(\f)$
let us define $\Psi$ and $\Psi'$ to be the reflections with
respect to $z$ and $\frac{Ta+Tb}{2}$, respectively. Consider a new
bijection on $E$ defined as a composition $S=\Psi T^{-1} \Psi' T$.
It is easy to check that $S \in W_E(\f \circ \f)$, $Sa=a$ and
$Sb=b$. We have:
\begin{eqnarray*}
2\|Tz-\frac{Ta+Tb}{2}\|_F &=& \|\Psi' Tz- Tz\|_F  \leq
\f(\|T^{-1}\Psi' Tz-T^{-1}Tz\|_E)\\ &=& \f(\|Sz-z\|_E) =
\f(\|Sz-\frac{Sa+Sb}{2}\|_E).
\end{eqnarray*}
Concluding
$$\lm(\f) \leq \frac{1}{2}\f(\lm(\f \circ \f)) \leq \f(\frac{\lm(\f^{ \circ 2})}{2}).$$
\noindent Hence:
$$
\lm(\f^{\circ 2^n}) \leq \f^{\circ 2^n}(\frac{\lm(\f^{\circ
2^{n+1}})}{2}).
$$
Applying the above formula recursively we obtain:

$$ \lm(\f)=\lm(\f^{\circ 1}) \leq \f^{\circ 1}(\frac{\lm(\f^{\circ
2})}{2}) \leq \f^{\circ 1}(\frac{1}{2} \f^{\circ
2}(\frac{\lm(\f^{\circ 4})}{2})) \leq \f^{\circ 1} \circ \f^{\circ
2}(\frac{\lm(\f^{\circ 4})}{4}) \leq \ldots $$ Finally:
$$\lm(\f) \leq \f^{\circ 1} \circ \f^{\circ 2} \ldots \f^{\circ 2^{n-1}}(\frac{\lm(\f^{\circ 2^n})}{2^n})=\f^{\circ (2^{n}-1)}(\frac{\lm(\f^{\circ 2^n})}{2^n}).$$
From the estimation (\ref{oszacowanie lambda}) we have

$$
\lm(\f) \leq \f^{\circ (2^{n+1}-1)}(\frac{\|a-b\|_E}{2^{n+1}}).
$$

\end{proof}

\section{Applications}
The result from the previous section gives us a very simple proof
of the main result from \cite{Gev1} as a consequence, which
answers the question of Hyers and Ulam. More precisely:
\begin{wn}\label{Hyers-Ulam}
Let $T$ be an $L$-isometry between Banach spaces $E$ and $F$ such
that $T(0)=0$. Then there exist constants $A$ and $B$, depending
on $L$ only, such that
$$\|T(\frac{a+b}{2})-\frac{Ta+Tb}{2}\| \leq A\sqrt{\|a-b\|}+B \textrm{ for all } a,b \in E.$$
As a corollary from that estimation, Gevirtz easily obtains
(relying on the result of Gruber) that the map $I:E \mapsto F$
defined as $Ix=\lim_{n \rightarrow \infty}\frac{T(2^nx)}{2^n}$ is
an isometry such that $\|Tx-Ix\| \leq 5L$ (later the constant was
improved to $2L$ which appears to be optimal).
\end{wn}
\begin{proof}
Let us first assume that $T$ is a $\f$-isometry for $\f(t)=t+L$.
Applying Theorem \ref{tw coarse mazur ulam} for $\f(t)=t+L$, we
obtain
$$\|T(\frac{a+b}{2})-\frac{Ta+Tb}{2}\| \leq
\frac{\|a-b\|}{2^{n+1}} + 2^{n+1}L.$$ Taking $n = \lfloor \log_2
\sqrt{\|a-b\|} \rfloor -1$ we have
$$ \|T(\frac{a+b}{2})-\frac{Ta+Tb}{2}\| =O(\sqrt{\|a-b\|})$$
as $\|a-b\| \rightarrow \infty$. By applying Fact \ref{surjekcja a
bijekcja}, we easily get the estimation for all $L$-isometries,
not only the bijective ones.
\end{proof}
For further applications of Theorem \ref{tw coarse mazur ulam} we
need the following simple observation:
\begin{lem}\label{lemma_on composition and integral}
Let $\f(t)=t+\e(t)$ where $\e :\Rz_+ \mapsto \Rz_+\setminus \{0\}$
is a non-decreasing function. Then $$\int^{\f^{\circ n}(t)}_{t}
\frac{1}{\e(x)} \dx \leq n.$$
\end{lem}
\begin{proof}
Let us notice that $\frac{1}{\e}$ is a non-increasing function,
hence
$$\int^{\f^{\circ n}(t)}_{t}
\frac{1}{\e(x)} \dx \leq \sum^{n-1}_{k=0} \frac{1}{\e(\f^{\circ
k}(t))} (\f^{\circ k+1}(t) - \f^{\circ k}(t))=n.$$
\end{proof}

We obtain the following:
\begin{wn}\label{oszacowanie dla zgrubnie lip}
Let $T : E \mapsto F$ be a $\f$-isometry for $\f(t)=(1+\e)t+L$
where $0<\e<0.2$.
 Then:
$$ \|T(\frac{a+b}{2})-\frac{Ta+Tb}{2}\| \leq 3\e\|a-b\| + \frac{4}{\e}L$$
for all $a,b \in E$.
\end{wn}
\noindent Let us explain that for $\e \geq 0.2$, we easily obtain
$$  \|T(\frac{a+b}{2})-\frac{Ta+Tb}{2}\| \leq \frac{1+\e}{2}\|a-b\| + \frac{L}{2}$$
which is a better estimate than the one from the above corollary
when $\|a-b\| \rightarrow \infty$. The above result is the most
interesting when $\e$ is close to 0 and $\|a-b\| \rightarrow
\infty$.
\begin{proof}
From Theorem \ref{tw coarse mazur ulam} we obtain that:
$$ \|T(\frac{a+b}{2})-\frac{Ta+Tb}{2}\| \leq \f^{\circ (k-1)}(\frac{d}{k})\leq \f^{\circ k}(\frac{d}{k})$$
where $k=2^{n+1}$ and $d=\|a-b\|$. From Lemma \ref{lemma_on
composition and integral} we get that $$\int^{\f^{\circ
k}(\frac{d}{k})}_{\frac{d}{k}} \frac{1}{\e x+L} \dx \leq k.$$
Hence $\f^{\circ k}(\frac{d}{k}) \leq \frac{e^{\e k}}{k}d
+\frac{1}{\e}(e^{\e k}-1)L$. Function $k \mapsto \frac{e^{\e
k}}{k}$ has its minimum at $k=\frac{1}{\e}$ which is $e \e$. Since
in our application $k=2^{n+1}$ we have to find $n$ so that
$2^{n+1}$ is as close to $\frac{1}{\e}$ as possible. For $\e < 0.2
< \frac{1}{\sqrt{2}}$ there exists \mbox{$n \in
[\log_{2}{\frac{1}{\e}}-1.5; \log_{2}{\frac{1}{\e}}-0.5] \cap
\mathbb{Z}_+$}. Hence $2^{n+1}=k \in
[\frac{1}{\sqrt{2}\e};\frac{\sqrt{2}}{\e}]$ and this interval
contains $\frac{1}{\e}$. Checking the values of $\frac{e^{\e
k}}{k}$ at the endpoints we obtain that $\f^{\circ k}(\frac{d}{k})
\leq 3\e d +\frac{4}{\e}L$.
\end{proof}
For the Lipschitz case ($L=0$) similar estimations can be found in
\cite{Vest1}. Vestfrid obtains the inequality $\|T(\frac{a+b}{2})
- \frac{Ta+Tb}{2}\| \leq 6\e\|a-b\|$. So one can see that the
above result improves the existing estimate as well as extends it
onto maps that are not necessarily continuous (\mbox{$L>0$}). By
applying the above Corollary we can also obtain some interesting
estimates for bi-Lipschitz maps between $\xi$-dense subspaces of
Banach spaces (nets in particular):

\begin{wn}\label{wn o sieciach}
Let us consider a $\f$-isometry $T: A \mapsto B$ from a $\xi_E$
dense set in Banach space $E$ onto a $\xi_F$ dense set in $F$,
where \mbox{$\f(t)=(1+\e)t$} and $0<\e<0.2$. Then for every $a,b
\in A$ and every $z \in A$ such that $\|\frac{a+b}{2}-z\| \leq
\xi_E$ we have:

$$ \|Tz-\frac{Ta+Tb}{2}\| \leq 3\e\|a-b\|+\frac{34(\xi_E+\xi_F)}{\e}.$$
\end{wn}
\begin{proof}
Using a simple Fact 1.5 from \cite{Gor} (or reasoning similarly as
in the proof of Fact \ref{surjekcja a bijekcja}) we obtain a map
$\widetilde{T}:E \mapsto F$ which is a $\f$-isometry for
$\f(t)=(1+\e)t+4\xi_F+3\xi_E$ and $\|\widetilde{T}x-Tx\| \leq
2\xi_F+2\xi_E$ for all $x \in A$. Let us take any $z \in A$ such
that $\|\frac{a+b}{2}-z\| \leq \xi_E$. Applying Corollary
\ref{oszacowanie dla zgrubnie lip} to the map $\widetilde{T}$, we
obtain the desired estimation.
\end{proof}

We will show now how Corollary \ref{oszacowanie dla zgrubnie lip}
allows us to obtain improvements on the constant $M_0$ and the
function $\delta$ in Theorem \ref{Banach-Stone ogolnie} for all
{\bf locally compact } spaces.

\begin{tw}\label{tw coarse banach stone}
Let $X$ and $Y$ be locally compact spaces. Consider a
\mbox{$\f$-isometry} \mbox{$T:C_0(X) \mapsto C_0(Y)$}  where
$\f(t)=Mt+L$ ($M \geq 1$, $L \geq 0$). If \mbox{$M < M_0=
\sqrt{\frac{16}{15}}$} then there exists a homeomorphism
$\varphi:X \mapsto Y$ and a continuous map $\lambda : X \mapsto
\{-1,1\}$ such that for every $f \in C_0(X)$

\begin{equation}\label{satbility eq}
\|Tf - If\| \leq 76(M-1)\|f\| + \D
\end{equation}
where $I$ is the isometry defined as
$If(y)=\lambda(\varphi^{-1}(y))f(\varphi^{-1}(y))$. The constant
$\D$ depends on $M$ and $L$ only.  Moreover, for $L=0$ we have
$\D=0$.
\end{tw}
As we can see the constant $M_0$ improves the result obtained by
Dutrieux and Kalton. However, more important is the estimation
\mbox{$\delta(M) \leq 76(M-1)$} that is far better then the
previously known, obtained by Jarosz in \cite{Jarosz}.
\begin{proof}
Let us assume that indeed $1<M<\sqrt{\frac{16}{15}}$. If $M=1$
then the above theorem easily follows from the mentioned solution
of the Hyers-Ulam problem (Corollary \ref{Hyers-Ulam}) and from
the Banach-Stone theorem. Let us first recall the construction of
the homeomorphism $\varphi$ and the function $\lambda$ from
\cite{Gor}.

In the construction, when dealing with topology of general
topological spaces, we use the notion of Moore-Smith convergence.
$\Sigma$ will always denote a directed set and whenever we write
$a_{\sigma} \rightarrow a$ we always mean $\lim_{\sigma \in
\Sigma}a_{\sigma}= a$.
\begin{df}
$(f^m_{\sigma})_{{\sigma} \in \Sigma} \subset C(X)$ is the m-peak
sequence at $x \in X$, for some directed set $\Sigma$ if
\begin{itemize}
\item $\|f^m_{\sigma}\|=|f^m_{\sigma}(x)|=m$ for all ${\sigma} \in
\Sigma$, \item $\lim_{{\sigma} \in \Sigma} f^m_{\sigma}|(X
\setminus U) \equiv 0$ uniformly for all open neighborhoods $U$ of
$x$.
\end{itemize}
The set of m-peak sequences at $x$ we denote by $P^X_m(x)$.
\end{df}
\begin{df} Let $D>0$ and $m>0$. We define the following:\\
$S^D_m(x)=\{y \in Y \textrm{; \;} \exists (f^m_{\sigma})_{{\sigma}
\in \Sigma} \in P^X_m(x) \textrm{\;} \exists y_{\sigma}
\rightarrow y \forall {\sigma} \in \Sigma \textrm{\;}
Tf^m_{\sigma}(y^m_{\sigma}) \geq Dm \textrm{ and }\\
T(-f^m_{\sigma})(y^m_{\sigma}) \leq -Dm\}$.\\
\end{df}

In \cite{Gor} author proves that for suitably chosen $D$ and $m$
we can define $\varphi(x)=S^D_m(x)$ that appears to be a
homeomorphism between $X$ and $Y$. In all the steps in \cite{Gor}
where we prove that $\varphi$ is a homeomorphism the only place
were compactness is crucial is Fact 2.4. We will modify its proof
using Corollary \ref{oszacowanie dla zgrubnie lip} so that it
works for the locally compact case.

\begin{fakt}\label{locally compact fact}
Let us consider $D$ such that $D = 14-13M$. There exists $m_0$
(depending on $M$ and $L$) such that for all $m > m_0$ we have
$S^D_m(x) \neq \emptyset$ for all $x \in X$. Moreover if $L=0$
then $m_0=0$.\\
\end{fakt}

\begin{proof}
Let us take any $(\widetilde{f}^m_{\sigma})_{\sigma \in \Sigma}
\in P^X_m(x)$ such that for all $\sigma \in \Sigma$
$\widetilde{f}^m_{\sigma}(x) = m$ and pick one $\sigma_0 \in
\Sigma$.  Let us define
$\widetilde{g}^m_{\sigma}=\frac{\widetilde{f}^m_{\sigma}+\widetilde{f}}{2}$
where $\widetilde{f}=\widetilde{f}^m_{\sigma_0}$. We have
$$\forall \sigma \in \Sigma \; \|T\widetilde{g}^m_{\sigma}-T(-\widetilde{g}^m_{\sigma})\| \geq
\frac{2}{M}m-L.$$ Hence $\forall \sigma \in \Sigma$ there exists
$y^m_{\sigma} \in Y$ such that
$|T\widetilde{g}^m_{\sigma}(y^m_{\sigma})-T(-\widetilde{g}^m_{\sigma})(y^m_{\sigma})|
\geq \frac{2}{M}m-L$. Let us observe that numbers
$T\widetilde{g}^m_{\sigma}(y^m_{\sigma})$ and
$T(-\widetilde{g}^m_{\sigma})(y^m_{\sigma})$ must be of different
signs. Assume the contrary. Since $\|T(\pm
\widetilde{g}^m_{\sigma})\| \leq Mm+L$ we have $Mm+L \geq
\frac{2}{M}m-L$ which is impossible for $m$ large enough, say $m >
m'_0$ (or for all $m>0$ if $L=0$), provided $\frac{2}{M}>M$ (that
is if $M < \sqrt{2}$). We can and we do assume that $\forall
\sigma \in \Sigma$ $T\widetilde{g}^m_{\sigma}(y^m_{\sigma}) \geq
0$ or $\forall \sigma \in \Sigma$
$T\widetilde{g}^m_{\sigma}(y^m_{\sigma}) \leq 0$. Let us define:
\begin{itemize}
\item If $\forall \sigma \in \Sigma \;
T\widetilde{g}^m_{\sigma}(y^m_{\sigma}) \geq 0$ then
$f^m_{\sigma}=\widetilde{f}^m_{\sigma}$, $f=\widetilde{f}$ and
$g^m_{\sigma}=\frac{f^m_{\sigma}+f}{2}$.

\item If $\forall \sigma \in \Sigma \;
T\widetilde{g}^m_{\sigma}(y^m_{\sigma}) \leq 0$ then
$f^m_{\sigma}=-\widetilde{f}^m_{\sigma}$, $f=-\widetilde{f}$ and
$g^m_{\sigma}=\frac{f^m_{\sigma}+f}{2}$.
\end{itemize}
Hence $Tg^m_{\sigma}(y^m_{\sigma})-T(-g^m_{\sigma})(y^m_{\sigma})
\geq \frac{2}{M}m-L$. Because $\|T(\pm g^m_{\sigma})\| \leq Mm+L$
then
$$Tg^m_{\sigma}(y^m_{\sigma}) \geq (\frac{2}{M}-M)m-2L$$
$$T(-g^m_{\sigma})(y^m_{\sigma}) \leq -(\frac{2}{M}-M)m+2L.$$
Since $g^m_{\sigma}=\frac{f^m_{\sigma}+f}{2}$ and by Corollary
\ref{oszacowanie dla zgrubnie lip} we obtain
$$\|T(\pm g^m_{\sigma})-\frac{T(\pm f^m_{\sigma})+T(\pm f)}{2}\|
\leq 3(M-1)m+\frac{4}{M-1}L.$$ Hence

$$Tf^m_{\sigma}(y^m_{\sigma}) \geq (\frac{4}{M}-9M+6)m-(5+\frac{8}{M-1})L,$$
$$T(-f^m_{\sigma})(y^m_{\sigma}) \leq -(\frac{4}{M}-9M+6)m+(5+\frac{8}{M-1})L,$$
$$Tf(y^m_{\sigma}) \geq (\frac{4}{M}-9M+6)m-(5+\frac{8}{M-1})L.$$
Let us consider $m_0 \geq m'_0$ such that
$$(\frac{4}{M}-9M+6)m_0-(5+\frac{8}{M-1})L \geq (14-13M)m$$
for all $m>m_0$ (we can do so since $\frac{4}{M}-9M+6>14-13M>0$
for all positive $M \neq 1$). By the compactness of the set
$$\{y \in Y: \; Tf(y) \geq (14-13M)m\}$$ for $m>m_0$
we can assume that $y^m_{\sigma} \rightarrow y \in S^D_m(x)$. Let
us notice that for $L=0$ we have $m_0=0$.
\end{proof}

Now the proof of Theorem \ref{tw coarse banach stone} is exactly
the same as the proof of Theorems 2.1 and Corollary 3.4 from
\cite{Gor}. Firstly, it is proven in \cite{Gor} Section 2 that
$\varphi(x)=S^D_m(x)$ is a homeomorphism for suitably chosen
$m>m_2$ (where $m_2=0$ if $L=0$) if
\begin{itemize}
\item[(i)] $D$ is so that $S^D_m(x) \neq \emptyset$ for all $x \in
X$;

\item[(ii)] $1-\e(M)M-\e(M)>0$ where $\e(M)=2M-1-D$.

\end{itemize}
In the compact case the condition (i) means that it is enough to
take \mbox{$D=4-3M<\frac{2}{M}-M$} (see Fact 2.4 in \cite{Gor}).
This, together with condition (ii), leads to a conclusion that
indeed $M<\sqrt{\frac{6}{5}}$. In the locally compact case we have
already shown (Fact \ref{locally compact fact}) that we can take
$D=14-13M$. Now the condition (ii) leads us to the inequality $M <
\sqrt{\frac{16}{15}}$.

For every $x \in X$ and $m>m_0$ let us define (following
\cite{Gor} Section 3)
$\lambda_m(x)=\frac{f^m_{\sigma}(x)}{|f^m_{\sigma}(x)|}$ where the
family $(f^m_{\sigma})_{{\sigma} \in \Sigma} \in P^X_m(x)$  is
such that:

\begin{itemize}

\item  $\forall \sigma_0,\sigma_1 \in \Sigma$
$\frac{f^m_{\sigma_0}(x)}{|f^m_{\sigma_0}(x)|}=\frac{f^m_{\sigma_1}(x)}{|f^m_{\sigma_1}(x)|}$
($\lambda_m(x)$ does not depend on $\sigma$).

\item $ \exists y_{\sigma} \rightarrow y \in S^D_m(x)$ such that
for every ${\sigma} \in \Sigma$ we have
$Tf^m_{\sigma}(y^m_{\sigma}) \geq Dm$ and \mbox{$
T(-f^m_{\sigma})(y^m_{\sigma}) \leq -Dm$}.
\end{itemize}
The existence of the above family for every $x \in X$ is exactly
what was shown in the proof of Fact \ref{locally compact fact}.
Let us say that the function $\lambda$ from the formulation of
Theorem \ref{tw coarse banach stone} is defined as $\lambda_m$ for
$m$ sufficiently large.

In order to prove (\ref{satbility eq}) it is enough to notice that
Fact 2.7 in \cite{Gor} works also for $X$ locally compact and
hence gives us the estimation
$$\big{|}|Tf(\varphi(x))| - |f(x)|\big{|} \leq \e(M)M \|f\| + \D=15(M^2-M) \|f\| + \D$$
for all $f \in C_0(X)$, $x \in X$ and some constant $\D$ depending
on $M$, $L$ and such that $\D=0$ if $L=0$.

Repeating the reasoning of Section 3 from \cite{Gor} for
$D=14-13M$ we obtain a slightly modified Fact 3.1 (only one
constant is changed):
\begin{fakt}\label{fakt o znakach}
Assume that $|f(x)| > 30(M-1)\|f\|$ and let $\|f\| = m$. Then for
$m>m_3$  ($m_3 \geq 0$ depends on $M$ and $L$ only), the sign of
$Tf(\varphi(x))$ is the same as the sign of $\lambda_m(x) f(x)$.
If $L=0$ then $m_3=0$.
\end{fakt}
As a consequence, reasoning in exactly the same way as in the
proof of Corollary 3.4 in \cite{Gor} we get (\ref{satbility eq})
where $\lambda \equiv \lambda_m$ for $m > m_3$. Summarizing the
proof let us just mention that having at hand Fact \ref{locally
compact fact} it is very easy to modify the reasoning from
\cite{Gor}. One should only keep in mind that this time
$D=14-13M$.

\end{proof}

\section{Final remarks}
Natural directions of further investigations and some open
problems arise from both of the above sections. First of all, as
we have already mentioned, it would be very interesting to see how
the result of Szankowski and Lindenstrauss follows from Theorem
\ref{tw coarse mazur ulam}. For instance, if we consider
$\e_T(t)=t+t^{\alpha}$ for $\alpha \in [0,1)$, by using Fact
\ref{lemma_on composition and integral} one can obtain that
$$\|T(\frac{a+b}{2}) - \frac{Ta+Tb}{2}\| =
O(\|a-b\|^{\frac{1}{2-\alpha}}) \textrm{ as $\|a-b\| \rightarrow
\infty$ }$$  which is sufficient to show that $Ix=\lim_{n
\rightarrow \infty}\frac{T(2^nx)}{2^n}$ is the required isometry
in Theorem \ref{LinSzan}.

Another interesting question concerns the expression
$\|T(\frac{a+b}{2}) - \frac{Ta+Tb}{2}\|$ and its optimal
estimation
 when $T$ is a $\f$-isometry
for $\f(t)=(1+\e)t$ and $\e \rightarrow 0$. We have already seen
that $\|T(\frac{a+b}{2}) - \frac{Ta+Tb}{2}\|=O(\e \|a-b\|)$ as $\e
\rightarrow 0$. It is very easy to show that this is everything
one can obtain in the general case. Indeed, as Vestfrid noticed in
\cite{Vest1}, consider $T: \Rz \mapsto \Rz$ defined as:
$$T(x)=\Bigg{\{}%
\begin{array}{cc}
  (1+\e)x \textrm{ if } x \geq 0 \\
  \frac{1}{1+\e}x \textrm{ if } x < 0.\\
\end{array}%
$$
However the exact value of the constant bellow remains unknown:

$$K=\limsup_{\e \rightarrow 0} K_{\e}$$ where $$K_{\e}=\sup \frac{\|T(\frac{a+b}{2}) - \frac{Ta+Tb}{2}\|}{\e\|a-b\|}.$$
Supremum is taken over all $T$ - $\f$-isometries between Banach
spaces, where $\f(t)=(1+\e)t$, and over all pairs of points $a
\neq b$ from the domain of $T$. The above example shows that $K
\geq 0.5$ and Corollary \ref{oszacowanie dla zgrubnie lip} shows
that $K \leq 3$. It is worth to notice that a simple analysis of
the proof of Corollary \ref{oszacowanie dla zgrubnie lip} gives us
that $\liminf_{\e \rightarrow 0} K_{\e} \leq e$.

Finally it is of a great interest to find the optimal constant
$M_0$ and the optimal estimation of $\delta$ in Theorem
\ref{Banach-Stone ogolnie}. In particular it is still unknown
whether the constant $M_0=\sqrt{2}$ is the optimal one or not.
However we skip the detailed discussion on this problem and we
direct the reader to the final section in \cite{Gor}.
%------------------------------------------------------------------------------------------------------------------


\begin{thebibliography}{99}


\bibitem{BeLin}  Y. Benyamini; J. Lindenstrauss : \emph{Geometric Nonlinear Functional Analysis},
Colloquium Publications, vol. 48, AMS (1993)

\bibitem{DutKal}  Y. Dutrieux; N. Kalton :
Perturbations of isometries between C(K)-spaces, \emph{ Studia
Math. } $\mathbf{ 166 }$ (2005), no. 2, 181--197

\bibitem{Cam}  M. Cambern :
On isomorphisms with small bound, \emph{ Proc. Amer. Math. Soc. }
$\mathbf{ 18 }$ (1967),  1062--1066

\bibitem{Gev1}  J. Gevirtz :
Stability of Isometries on Banach Spaces, \emph{ Proc. Amer. Math.
Soc. } $\mathbf{ 89 }$ (1983), no. 4,  633--636

\bibitem{Gor} R. G\'orak : Coarse version of the Banach-Stone theorem, \emph{ Journal
of Mathematical Analysis and Applications } $\mathbf{ 377 }$
(2011), 406--413

\bibitem{Kal} N. Kalton: The nonlinear geometry of Banach spaces, \emph{ Rev.
Mat. Complut. } $\mathbf{ 21 }$ (2008), no. 1, 7–-60

\bibitem{LinSzan} J. Lindenstrauss and A. Szankowski : Non-linear
perturbations of isometries, \emph{ Astérisque } $\mathbf{ 131 }$
(1985), 357–-371

\bibitem{Rassias}  T. M. Rassias :
Isometries and Approximate Isometries, \emph{ Int. J. Math. Math.
Sci. } $\mathbf{ 25 }$ (2001), no. 2, 73--91

\bibitem{Jarosz}  K. Jarosz :
Nonlinera generalization of Banach-Stone theorem, \emph{ Studia
Math. } $\mathbf{ 93 }$ (1989), no. 2, 97--107

\bibitem{Vaisala} J. V\"{a}is\"{a}l\"{a} :  A proof of the Mazur-Ulam theorem, \emph{ Amer. Math.
Monthly } $\mathbf{ 110 }$ (2003), no. 7, 633–-635

\bibitem{Vest1}  I. Vestfrid :
Affine properties and injectivity of quasi isometries, \emph{ Isr.
J. Math. } $\mathbf{ 141 }$ (2004), 185--210

\end{thebibliography}
\end{document}